\documentclass{article}%
\usepackage{amsmath}
\usepackage{amsfonts}
\usepackage{amssymb}
\usepackage{graphicx, enumerate}
\usepackage[numbers,sort&compress]{natbib}%
\newtheorem{theorem}{Theorem}[section]

\newtheorem{corollary}[theorem]{Corollary}

\newtheorem{definition}[theorem]{Definition}

\newtheorem{example}[theorem]{Example}

\newtheorem{lemma}[theorem]{Lemma}

\newtheorem{proposition}[theorem]{Proposition}
\newtheorem{remark}[theorem]{Remark}

\newenvironment{proof}[1][Proof]{\noindent \textbf{#1.} }{\  \rule{0.5em}{0.5em}}
\numberwithin{equation}{section}
\def\Tr{\mathrm{Tr}}
\def\var{\mathrm{var}}

\begin{document}

\title{Reflected Backward Stochastic Difference Equations and Optimal Stopping
Problems under $g$-expectation}
\author{Lifen An \thanks{School of Mathematics, Shandong University.}
\and Samuel N. Cohen\thanks{Mathematical Institute, University of Oxford, samuel.cohen@maths.ox.ac.uk. Research supported by the Oxford--Man Institute for Quantitative Finance.}
\and Shaolin Ji\thanks{Qilu Institute of Finance, Shandong University, jsl@sdu.edu.cn.}}
\maketitle
\date{}

\textbf{Abstract}: In this paper, we study reflected backward stochastic
difference equations (RBSDEs for short) with finitely many states in discrete
time. The general existence and uniqueness result, as well as comparison
theorems for the solutions, are established under mild assumptions. The
connections between RBSDEs and optimal stopping problems are also given. Then
we apply the obtained results to explore optimal stopping problems under
$g$-expectation. Finally, we study the pricing of American contingent claims
in our context.

\bigskip


\bigskip

{ \textbf{Keywords}: backward stochastic difference equations
(BSDEs), reflected BSDEs (RBSDEs), optimal stopping, g-expectation, American
contingent claims \newline}

\section{Introduction}

The theory of nonlinear backward stochastic differential equations was first
introduced by Pardoux and Peng \cite{r17}. Over the past twenty years,
backward stochastic differential equations have been widely used in
mathematical finance, stochastic control and other fields. By analogy with the
equations in continuous time, Cohen and Elliott \cite{r8} consider backward
stochastic difference equations {\ (BSDEs)} on spaces related to discrete
time, finite state processes, establishing fundamental results including the
comparison theorem etc. These are studied as entities in their own right, not
as approximations to continuous BSDEs, as in \cite{r1,r2,r15, CherSta}. For
deeper discussion, the readers may refer to \cite{r8,CE,r9,CH,CH1}.

{\ The general theory of reflected backward stochastic differential equations
was studied by El Karoui et al. \cite{r10}. They considered the case where the
solution is forced to stay above a given stochastic process (called the
obstacle) and introduced an increasing process which pushes the solution to
satisfy this condition. This important theory has been applied to optimal
stopping problems (see \cite{r11}) as well as to problems in financial markets
and other related fields (see \cite{BHM,CK,H,HLM1997,HT,LX,LM,Matoussi}). For
this reason, it is interesting to explore reflected backward stochastic
difference equations (RBSDEs for short) in the framework of \cite{r8}, as well
as some applications in optimal stopping problems in discrete time. }

The RBSDE is formulated in detail in section 2. To associate the RBSDE
solution with the classical Skorohod problem, as is done in the continuous
time case (see \cite{r10}), we first prove that the Skorohod lemma remains
valid in our framework. Using the Skorohod lemma, the increasing process which
forces the solution is expressed as a supremum. Then we give the main results
of this paper including the comparison theorem and the existence and
uniqueness theorem. The proof of the comparison theorem is similar to that for
nonreflected BSDE's in \cite{r8}. Existence of solutions is established by
penalization of the constraints. Moreover we show that the solution of an
RBSDE corresponds to the value of an optimal stopping problem. We also show
that the solution of the RBSDE in which the coefficient $f$ is a concave (or
convex) function corresponds to the value function of a mixed optimal
stopping--optimal stochastic control problem.

With some limitations on the generator $g$, a BSDE can be used to define a
nonlinear expectation $\mathcal{E}^{g}[\xi]:=Y_{0}$, which is called
$g$-expectation (see \cite{r18}). A notable property of $g$-expectation is its
time consistency, namely the property that the conditional expectation
$\mathcal{E}^{g}[\xi|\mathcal{F}_{t}]$ can be well-defined. Furthermore, it
was proved that a dominated and time-consistent nonlinear expectation can be
represented as the solution of a BSDE (see \cite{CHMP, SCGenDom} in continuous
time, \cite{r8} in discrete time). So it is interesting to study optimal
stopping problems under $g$-expectation. In section 3, we first study the
$g$-expectation theory on spaces related to discrete time, finite state
processes. The Doob-Mayer decomposition theorem and optional sampling theorem
are obtained. Then we associate $g$-martingales with multiple prior
martingales which were introduced by Riedel in \cite{r22}. We finally show
that RBSDE is a convenient tool to solve some optimal stopping problems under
$g$-expectation.

In section 4, we apply the obtained results to study the pricing of American
contingent claims in an incomplete financial market.

\section{ RBSDEs}

Following \cite{r8}, we consider an underlying discrete time,
finite state process ${X}$ {which takes values in the
standard basis vectors of $\mathbb{R}^{m}$, where $m$ is the number of states of the
process $X$ . In more detail, for each
$t\in\mathcal{N}\triangleq\{0,1,2,...,T\}$, $X_{t}\in\{e_{1},...,e_{m}\}$,
where $T>0$  is a finite deterministic terminal time,
 $e_{i}=(0,0,...,0,1,0,...,0)^{\ast}\in \mathbb{R}^{m}$, and $[\cdot
]^{\ast}$ denotes vector transposition.  We note in passing that $E[X_t]$ is a vector containing the probabilities of $X_t$ being in each of its states.

Consider  a filtered probability space $(\Omega,\mathcal{F}%
,\{{\mathcal{F}_{t}}\}_{0\leq t\leq T},P),$ where $\mathcal{F}_{t}$ is the
completion of the $\sigma$-algebra generated by the process $X$
 up to time ${t}$ and $\mathcal{F}=\mathcal{F}_{T}$.

Define
\[
 M_{t}=X_{t}-E[X_{t}\mid\mathcal{F}_{t-1}],\ t=1,...,T.
\]
 $M$ is a martingale difference process taking values in $\mathbb{R}^m$, and  we define the following equivalence relation.
\begin{definition}\label{def:simM}
We define $Z^{1}\sim_{M}Z^{2}$  whenever $\Vert Z^{1}-Z^{2}\Vert_{M}^{2}=0$ where
\begin{align*}
\Vert Z\Vert_{M}^{2}  &  \triangleq E\Tr\Big[\sum_{0\leq u<T}Z_{u}^{\ast}\cdot
E[M_{u+1}M_{u+1}^{\ast}|\mathcal{F}_{u}]\cdot Z_{u}\Big]\\
&  =\sum_{0\leq u<T}\Tr E\Big[(Z_{u}^{\ast}M_{u+1})(Z_{u}^{\ast}M_{u+1})^{\ast}\Big].
\end{align*}
\end{definition}
From \cite[Theorem 1]{r8}, we have the following martingale representation theorem.

\begin{theorem}\label{MRT}
 For any $\{\mathcal{F}_t\}$-adapted $\mathbb{R}^K$-valued martingale $L$, there exists an adapted $\mathbb{R}^{K\times N}$ valued process $Z$ such that
\[L_t = L_0 + \sum_{0\leq u <t} Z_u M_{u+1},\]
Moreover, this process is unique up to equivalence $\sim_{M}$.
\end{theorem}

The general form of a backward stochastic difference equation in \cite{r8} is
for any $0\leq t\leq T,$
\begin{equation}
Y_{t}=\xi+\sum_{t\leq u<T}{f(u,Y_{u},Z_{u})}-\sum_{t\leq u<T}{Z_{u}^{\ast
}M_{u+1}},\qquad{ P-a.s.} \label{bsde-1}%
\end{equation}
where $\xi$ is an $\mathbb{R}$-valued $\mathcal{F}_{T}$-measurable terminal condition
and $f$ an adapted map $f:\Omega\times\{0,1,...,T\} \times \mathbb{R}\times
\mathbb{R}^{m}\rightarrow \mathbb{R}.$ The solution $(Y_{t},Z_{t})$ is adapted to the filtration
$\{{\mathcal{F}_{t}}\}$ and takes values in $\mathbb{R}\times \mathbb{R}^{m}$. We also assume\footnote{Note that since $X_{t}$ only takes finite states, it is clear that $L^{1}(\mathcal{F}_{t};\mathbb{R}^{m})=L^{\infty}(\mathcal{F}_{t};\mathbb{R}^{m})$.}
that $(Y_{t},Z_{t})\in L^{1}(\mathcal{F}_{t};\mathbb{R})\times L^{1}(\mathcal{F}_{t};\mathbb{R}^{m})$ for all $t,$ $\xi\in L^{1}(\mathcal{F}_{T};\mathbb{R})$ and $f(t,y,z)\in L^{1}(\mathcal{F}_{t};\mathbb{R})$ for all $t$ and $(y,z)\in \mathbb{R}\times \mathbb{R}^{m}$.

Theorem 2 in  \cite{r8} gives the following general existence result.

\begin{theorem}[BSDE existence and uniqueness]\label{thm:BSDEexist}
The BSDE (\ref{bsde-1}) has a unique adapted
solution $(Y_{t},Z_{t})$ if and only if $f$ satisfies the following two assumptions
\begin{enumerate}[(i)]
\item For any $Y$, if $Z^{1}\sim_{M}Z^{2}$, then $f(t,Y_{t},Z_{t}^{1})=f(t,Y_{t},Z_{t}^{2})$ $P$-$a.s.$
for all $t$;
\item For any $z\in \mathbb{R}^{m}$, all $t$ and $P$-almost all $\omega$,the map
\[y{ \mapsto y-f(t,y,z)}\]
 is a bijection $\mathbb{R}\rightarrow \mathbb{R}$.
\end{enumerate}
\end{theorem}

We now consider RBSDEs in this setting.

\begin{definition}[Reflected BSDE]
A triple $(\xi,f,S)$ is called `standard data' for an RBSDE if

\begin{enumerate}[(i)]
\item $\xi\in L^{1}(\mathcal{F}_{T};\mathbb{R})$ ;
\item The map $f(\cdot,y,z)$ is an adapted process for any $(y,z)\in \mathbb{R}\times \mathbb{R}^{m}$;
\item The obstacle process $\{S_{t},0\leq t\leq T\}$ is real-valued, adapted and such that $S_{T}\leq\xi$ $P$-a.s.
\end{enumerate}
\end{definition}

\begin{definition}\label{FS_BSDE}

 A solution of RBSDE with standard data $(\xi,f,S)$ is a triple
$\{(Y_{t},Z_{t},K_{t})\}_{0\leq t\leq T}$ of adapted processes taking values in
$\mathbb{R}\times \mathbb{R}^{m}\times \mathbb{R}$ such that, for all $0\leq t\leq T,$
\begin{enumerate}[(i)]
 \item
\begin{equation}
{ Y_{t}=\xi+\sum_{t\leq u<T}{f(u,Y_{u},Z_{u})}+K_{T}-K_{t}%
-\sum_{t\leq u<T}{Z_{u}^{\ast}M_{u+1}}},\; P\text{-a.s.};
\label{fs-bsde1}
\end{equation}

\item ${ Y_{t}\in L^{1}(\mathcal{F}%
_{t};\mathbb{R}),Z_{t}\in L^{1}(\mathcal{F}_{t};\mathbb{R}^{m})\ and\ K_{t}\in}L^{1}%
(\mathcal{F}_{t};\mathbb{R})$;

\item $Y_{t}\geq S_{t}$  $P$-a.s.

\item $\{K_{t}\}$ is increasing in $t$, $K_{0}=0$ and
\[
\sum_{0\leq t\leq T}(Y_{t}-S_{t})(K_{t+1}-K_{t})=0,\;P-a.s.
\]
\end{enumerate}

\end{definition}

\subsection{Skorohod lemma and a priori estimate}

To determine solutions to an RBSDE, it is first useful to consider how it is connected to the Skorohod problem (see \cite{r21}) in the
discrete time case. This will then allow us to obtain a priori estimates on the behaviour of solutions to our RBSDEs.

\begin{lemma}[Skorohod problem]
\label{skorohod}Let $y$ be a real-valued function on $\{0,1,...,T\}$ such that
$y(0)\geq0$. There exists a unique pair $(v,g)$ of functions on
$\{0,1,...,T\}$ such that, for all $t\in\{0,1,...,T\},$
\begin{enumerate}[(i)]
\item $v(t)=y(t)+g(t)$;
\item  $v(t)$ is non-negative;
\item $g(t)$ is increasing, vanishing at zero and
\[
\sum_{1\leq t\leq T}v(t)(g(t)-g(t-1))=0.
\]
\end{enumerate}

The function $g$ is moreover given by
\[
g(t)=\sup_{s\leq t}(-y(s)\vee 0).
\]

\end{lemma}

\begin{proof}
We first claim that the pair $(g,v)$ defined by
\[
g(t)=\sup_{s\leq t}(-y(s)\vee0),\qquad v(t)=y(t)+g(t)
\]
satisfies properties (\emph{i}) through (\emph{iii}).

To prove the uniqueness of the pair $(g,v)$, we suppose that $(\hat{g},\hat
{v})$ is another pair which satisfies (\emph{i}) through (\emph{iii}). Then
$v-\hat{v}=g-\hat{g}.$ Note that $g(0)=\hat{g}(0)=0$ and consequently
$v(0)-\hat{v}(0)=0$. Thus,
\begin{align*}
&(v(t)-\hat{v}(t))^{2}  \\
&  =\sum_{1\leq s\leq t}[(v(s)-\hat{v}(s))^{2} -(v(s-1)-\hat{v}(s-1))^{2}]\\
&  =\sum_{1\leq s\leq t}[(v(s)-\hat{v}(s))+(v(s-1)-\hat{v}(s-1))][(g(s)-\hat{g}(s))  -(g(s-1)-\hat{g}(s-1))]\\
&  =\sum_{1\leq s\leq t}(v(s)-\hat{v}(s))[(g(s)-\hat{g}(s))-(g(s-1)-\hat{g}(s-1))]\\
&  \qquad +\sum_{1\leq s\leq t}(g(s-1)-\hat{g}(s-1))(g(s)-\hat{g}(s))-\sum_{1\leq s\leq t}(g(s-1)-\hat{g}(s-1))^{2}\\
&  =-\sum_{1\leq s\leq t}v(s)(\hat{g}(s)-\hat{g}(s-1))-\sum_{1\leq s\leq t}\hat{v}(s)(g(s)-g(s-1))\\
&  \qquad -\sum_{1\leq s\leq t}(g(s-1)-\hat{g}(s-1))^{2}+\sum_{1\leq s\leq t}(g(s)-\hat{g}(s))(g(s-1)-\hat{g}(s-1))\\
&  \leq-\sum_{1\leq s\leq t}v(s)(\hat{g}(s)-\hat{g}(s-1))-\sum_{1\leq s\leq t}\hat{v}(s)(g(s)-g(s-1))\\
&  \qquad -\sum_{1\leq s\leq t}\frac{(g(s-1)-\hat{g}(s-1))^{2}}{2}+\sum_{1\leq s\leq t}\frac{(g(s)-\hat{g}(s))^{2}}{2}.
\end{align*}

As $v-\hat{v}=g-\hat{g}$, we have that
\[
\frac{(v(t)-\hat{v}(t))^{2}}{2}\leq-\sum_{1\leq s\leq t}v(s)(\hat{g}%
(s)-\hat{g}(s-1))-\sum_{1\leq s\leq t}\hat{v}(s)(g(s)-g(s-1))\leq0.
\]
Hence $v(t)=\hat{v}(t)$ and consequently $g(t)=\hat{g}(t)$.
\end{proof}

Using Lemma \ref{skorohod}, we obtain the following estimate for solutions of RBSDEs.

\begin{proposition}
Let $\{(Y_{t},Z_{t},K_{t}),0\leq t\leq T\}$ be a solution of the RBSDE
(\ref{fs-bsde1}). Then for each $t\in\{0,1,...,T\}$,
\[K_{T}-K_{t}=\underset{t\leq u\leq T}{\sup}\Big(\xi+\sum_{u\leq s<T}f(s,Y_{s},Z_{s})-\sum_{u\leq s<T}Z_{s}^{\ast}M_{s+1}-S_{u}\Big)^{-}.
\]
\end{proposition}

\begin{proof}
Set
\[
y_{t}=\xi+\sum_{T-t\leq s<T}f(s,Y_{s},Z_{s})-\sum_{T-t\leq s<T}Z_{s}^{\ast
}M_{s+1}-S_{T-t}.
\]
 Then $y_{0}=\xi-S_{T}\geq0.$ Note that
\[
 Y_{T-t}(\omega)-S_{T-t}(\omega)=y_{t}+K_{T}(\omega)-K_{T-t}(\omega).
\]
From the properties of the RBSDE, we can see that
\[(v(t),g(t))=(Y_{T-t}(\omega)-S_{T-t}(\omega),K_{T}(\omega)-K_{T-t}(\omega))\qquad 0\leq t\leq T,\]
 is a solution of the above Skorohod problem.  By Lemma \ref{skorohod}, this solution is unique, and we can write
\[
K_{T}-K_{T-t}=\underset{0\leq u\leq t}{\sup}\Big(\xi+\sum_{T-u\leq s<T}f(s,Y_{s},Z_{s})-\sum_{T-u\leq s<T}Z_{s}^{\ast}M_{s+1}-S_{T-u}\Big)^{-}.
\]
This completes the proof.
\end{proof}

\subsection{Comparison theorem}

We now present a comparison theorem for RBSDEs. Given $\mathcal{F}_{t}$, let $\mathcal{Q}_{t}$ denote the $\mathcal{F}_{t}$-measurable set of indices of possible values of $X_{t+1}$,
 i.e.
\begin{equation}\label{Qtset}
\mathcal{Q}_{t}\triangleq\{i:P(X_{t+1}=e_{i}\mid\mathcal{F}_{t})>0\}.
\end{equation}

\begin{theorem}[Comparison Theorem]\label{Comparison Theorem}
Consider two RBSDEs
with standard data $(\xi^{1},f^{1},S^{1})$ and $(\xi^{2},f^{2},S^{2})$
respectively. Let $(Y^{1},Z^{1},K^{1})$ and $(Y^{2},Z^{2},K^{2})$ be the
associated solutions. Suppose the following conditions hold $P$-a.s. for all $t$
\begin{enumerate}[(i)]
 \item $\xi^{1}\geq\xi^{2}$,
 \item $f^{1}(t,Y_{t}^{2},Z_{t}^{2})\geq f^{2}(t,Y_{t}^{2},Z_{t}^{2})$
 \item $S_{t}^{1}\geq S_{t}^{2}$,
 \item $f^{1}(t,Y_{t}^{2},Z_{t}^{1})-f^{1}(t,Y_{t}^{2} ,Z_{t}^{2})\geq\underset{i\in\mathcal{Q}_{t}}{\min}\{(Z_{t}^{1}-Z_{t}^{2})^{\ast}(e_{i}-E[X_{t+1}\mid\mathcal{F}_{t}])\}$,
 \item if $Y_{t}^{1}-f^{1}(t,Y_{t}^{1},Z_{t}^{1})\geq Y_{t}^{2}-f^{1}(t,Y_{t}^{2},Z_{t}^{1})$, then $Y_{t}^{1}\geq Y_{t}^{2}$.
\end{enumerate}
Then it is true that, for all $t$,
\[
Y_{t}^{1}\geq Y_{t}^{2}\quad P\text{-a.s.}
\]

\end{theorem}

\begin{proof}
It is clear that $Y_{T}^{1}-Y_{T}^{2}=\xi^{1}-\xi^{2}\geq0$ $P$-a.s.  For an arbitrary $0\leq t<T$, suppose that $Y_{t+1}^{1}-Y_{t+1}^{2}\geq0$ $ P$-a.s. We then have
\begin{equation}\label{CT-bsde1}
\begin{split}
 Y_{t+1}^{1}-Y_{t+1}^{2} &=  Y_{t}^{1}-Y_{t}^{2}-f^{1}(t,Y_{t}^{1},Z_{t}^{1})+f^{2}(t,Y_{t}^{2},Z_{t}^{2})+(Z_{t}^{1}-Z_{t}^{2})^{\ast}M_{t+1}\\
& \qquad -(K_{t+1}^{1}-K_{t}^{1})+(K_{t+1}^{2}-K_{t}^{2})\\
&\geq  0.
\end{split}
\end{equation}
Since $M_{t+1}=X_{t+1}-E[X_{t+1}\mid\mathcal{F}_{t}]$ and $X_{t+1}$ almost surely takes values in $\mathcal{Q}_{t}$,
\[
\begin{split}
& Y_{t}^{1}-Y_{t}^{2}-(K_{t+1}^{1}-K_{t}^{1})+(K_{t+1}^{2}-K_{t}^{2})\\
&\geq  f^{1}(t,Y_{t}^{1},Z_{t}^{1})-f^{2}(t,Y_{t}^{2},Z_{t}^{2})-\underset{i\in\mathcal{Q}_{t}}{\min}\{(Z_{t}^{1}-Z_{t}^{2})^{\ast}
(e_{i}-E[X_{t+1}\mid\mathcal{F}_{t})\}.
\end{split}
\]
By assumptions (ii) and (iv), we obtain
\begin{equation}
\begin{split}
& Y_{t}^{1}-Y_{t}^{2}-f^{1}(t,Y_{t}^{1},Z_{t}^{1})+f^{1}(t,Y_{t}^{2},Z_{t}^{1})-(K_{t+1}^{1}-K_{t}^{1})+(K_{t+1}^{2}-K_{t}^{2})\\
& \geq  f^{1}(t,Y_{t}^{2},Z_{t}^{2})-f^{2}(t,Y_{t}^{2},Z_{t}^{2})+f^{1}(t,Y_{t}^{2},Z_{t}^{1})-f^{1}(t,Y_{t}^{2},Z_{t}^{2})\\
& \qquad -\underset{i\in\mathcal{Q}_{t}}{\min}\{(Z_{t}^{1}-Z_{t}^{2})^{\ast}(e_{i}-E[X_{t+1}\mid\mathcal{F}_{t}])\}\\
&\geq  0.
\end{split}
\label{CT-bsde2}
\end{equation}

Set
\[
\mathcal{A}\triangleq\{ \omega\mid Y_{t}^{1}(\omega)<Y_{t}^{2}(\omega)\}.
\]
We know that $S_{t}^{2}\leq S_{t}^{1}\leq Y_{t}^{1}<Y_{t}^{2}$ on
$\mathcal{A}$, which yields that $K_{t+1}^{2}-K_{t}^{2}=0$ on
$\mathcal{A}$. Therefore,
\[
Y_{t}^{1}-Y_{t}^{2}-f^{1}(t,Y_{t}^{1},Z_{t}^{1})+f^{1}(t,Y_{t}^{2},Z_{t}%
^{1})\geq0\; \text{on}\ \mathcal{A}.
\]
But, by assumption (v), the above inequality implies
$Y_{t}^{1}\geq Y_{t}^{2}$ on $\mathcal{A}$. Thus, we deduce that $P(\mathcal{A})=0$ and
\[
Y_{t}^{1}\geq Y_{t}^{2}\quad P-a.s.
\]
This completes the proof.
\end{proof}

\begin{remark}
If the map $y\mapsto y-f(\omega, t,y,z)$ is strictly increasing in $y$ for all $t$ and $z$ and $P$-almost all $\omega$, then assumption (v) holds.
\end{remark}

\begin{corollary}
\label{Corollary Comparison Theorem} Suppose the assumptions of Theorem \ref%
{Comparison Theorem} hold. Set $t\in \{0,1,...,T\}$. If we also know that $%
Y_{s}^{1}=Y_{s}^{2}$ for all $s\in \{0,1,...,t\}$, then $K_{s}^{1}\leq
K_{s}^{2}$ $P$-a.s. for all $s\in \{0,1,...,(t+1)\wedge T\}$ and $%
K_{s}^{1}-K_{s}^{2}$ is decreasing in $s$. Moreover, if $\xi ^{1}=\xi ^{2}$
and $f^{1}=f^{2}$ $P$-a.s., then $K_{t}^{1}=K_{t}^{2}$ $P$-a.s. for all $%
t\in \{0,1,...,T\}$.
\end{corollary}

\begin{proof}
By (\ref{CT-bsde2}), we have
\begin{equation*}
K_{s+1}^{1}-K_{s}^{1}\leq K_{s+1}^{2}-K_{s}^{2}.
\end{equation*}%
Then
\begin{equation*}
K_{s+1}^{1}-K_{s+1}^{2}\leq K_{s}^{1}-K_{s}^{2},
\end{equation*}%
that is, $K_{s}^{1}-K_{s}^{2}$ is decreasing in $s.$ Since $%
K_{0}^{1}=K_{0}^{2}=0$, we obtain that $K_{s}^{1}\leq K_{s}^{2}$ $P$-a.s.

Moreover, if we also have $\xi^{1}=\xi^{2}$ and $f^{1}=f^{2}$ $P$-a.s., then it is easy to see that
\[
K_{t}^{2}\leq K_{t}^{1}\quad P\text{-a.s.}
\]
which completes the proof.
\end{proof}

The following example shows that Theorem \ref{Comparison Theorem}
fails when assumption (iv) does not hold.

\begin{example}
For simplicity, suppose $T=1$. Consider two RBSDEs with standard data
$(\xi^{1},f^{1},S^{1})$ and $(\xi^{2},f^{2},S^{2})$ respectively which satisfy
the assumptions of Theorem \ref{Existence Uniqueness} in the following
section. Let $\xi^{1}=\xi^{2}$, $f^{1}=f^{2}=f$ and $S^{1}=S^{2}$ and the map
$y-f(y,z)$ be strictly increasing in $y$. By Theorem
\ref{Existence Uniqueness}, we have $Y_{0}^{1}=Y_{0}^{2}$, $K_{0}^{1}%
=K_{0}^{2}$ and $K_{1}^{1}=K_{1}^{2}$ $P-a.s.$

Suppose that assumption (iv) of Theorem \ref{Comparison Theorem} does
not hold. In particular, we have
\[
f(0,Y_{0}^{2},Z_{0}^{1})-f(0,Y_{0}^{2},Z_{0}^{2})<\underset{i\in
\mathcal{Q}_{t}}{\min}\{(Z_{0}^{1}-Z_{0}^{2})^{\ast}(e_{i}-E[X_{1}%
\mid\mathcal{F}_{0}])\}.
\]
Then we have
\[
\begin{split}
0&=  Y_{1}^{1}-Y_{1}^{2}\\
&= Y_{0}^{1}-Y_{0}^{2}-f(0,Y_{0}^{1},Z_{0}^{1})+f(0,Y_{0}^{2},Z_{0}%
^{2})+(Z_{0}^{1}-Z_{0}^{2})^{\ast}M_{1}\\
&\qquad -(K_{1}^{1}-K_{0}^{1})+(K_{1}^{2}-K_{0}^{2})\\
&>  Y_{0}^{1}-Y_{0}^{2}-f(0,Y_{0}^{1},Z_{0}^{1})+f(0,Y_{0}^{2},Z_{0}%
^{1})-(K_{1}^{1}-K_{0}^{1})+(K_{1}^{2}-K_{0}^{2}).
\end{split}
\]
{ It follows that}%
\[\begin{split}
0&=  (K_{1}^{1}-K_{0}^{1})-(K_{1}^{2}-K_{0}^{2})\\
&>  { Y_{0}^{1}-Y_{0}^{2}-f(0,Y_{0}^{1},Z_{0}^{1})+f(0,Y_{0}^{2},Z_{0}^{1}).}
\end{split}
\]
As the map $y\mapsto y-f(y,z)$ is strictly increasing, we deduce $Y_{0}^{1}<Y_{0}^{2}$, contradicting the conclusion of Theorem
\ref{Comparison Theorem}.
\end{example}

\subsection{Existence and uniqueness}
In this subsection, we will explore the existence and uniqueness
of solutions of RBSDE basing on approximation via penalization in \cite{r10}
as well as the comparison theorem obtained in \cite{r8}.

First, we recall the comparison theorem in \cite{r8}.

\begin{theorem}
\label{Comparison Theorem1} Consider two BSDEs (\ref{bsde-1}) with standard
data $(\xi^{1},f^{1})$ and $(\xi^{2},f^{2})$ respectively. Suppose
$(Y^{1},Z^{1})$ and $(Y^{2},Z^{2})$ are the associated solutions, and the
following conditions also hold $P$-a.s. for all $t\in\{0,1,...,T\}$,
\begin{enumerate}[(i)]
 \item $\xi^{1}\geq\xi^{2}$,
 \item $f^{1}(t,Y_{t}^{2},Z_{t}^{2})\geq f^{2}(t,Y_{t}%
^{2},Z_{t}^{2})$,
\item $f^{1}(t,Y_{t}^{2},Z_{t}^{1})-f^{1}(t,Y_{t}%
^{2},Z_{t}^{2})\geq\underset{i\in\mathcal{Q}_{t}}{\min}\{[Z_{t}^{1}-Z_{t}%
^{2}]^{\ast}(e_{i}-E[X_{t+1}\mid\mathcal{F}_{t}])\}$,
\item if $Y_{t}^{1}-f^{1}(t,Y_{t}^{1},Z_{t}^{1})\geq
Y_{t}^{2}-f^{1}(t,Y_{t}^{2},Z_{t}^{1})$ then $Y_{t}^{1}\geq Y_{t}^{2}.$
\end{enumerate}
Then it is true that, for all $t$,
\[
Y_{t}^{1}\geq Y_{t}^{2}\quad P\text{-a.s.}
\]
\end{theorem}

\begin{corollary}
\label{Comparison Theorem11} Suppose Theorem \ref{Comparison Theorem1} holds, and furthermore
\begin{itemize}
\item at least one of inequalities (i) and (ii) is strict,
 \item inequality (iii) is strict unless both sides are zero, and
 \item the map $y\mapsto y-f(t, y, Z^1)$ is strictly increasing (guaranteeing inequality (iv)).
\end{itemize}
Then we have $Y_{t}^{1}>Y_{t}^{2}$ $P$-a.s. for all $t$.
\end{corollary}

\begin{proof}
By Theorem \ref{Comparison Theorem1}, we have $Y_{t}^{1}\geq Y_{t}^{2}$
$P$-a.s. Then, by the same arguments as in Theorem \ref{Comparison Theorem},
we obtain
\[
Y_{t}^{1}-Y_{t}^{2}-f^{1}(t,Y_{t}^{1},Z_{t}^{1})+f^{1}(t,Y_{t}^{2},Z_{t}^{1})>0.
\]
It follows that $Y_{t}^{1}>Y_{t}^{2}$ $P$-a.s. This completes the proof.
\end{proof}

\begin{theorem}
\label{Existence Uniqueness} Consider a RBSDE {(\ref{fs-bsde1})}
with standard data $(\xi,f,S)$. The map $f$ satisfies the following two assumptions $P$-a.s. for all $t$
\begin{enumerate}[(i)]
\item  For any $Y$, if $Z^{1}\sim_{M}Z^{2}$, then $f(t,Y_{t},Z_{t}^{1})=f(t,Y_{t},Z_{t}^{2})$
\item For any $z\in R^{m}$ the map $y\mapsto y-f(t,y,z)$ is strictly increasing and continuous in $y$
\end{enumerate}
Then there exists an adapted solution $(Y,Z,K)$ for the RBSDE
(\ref{fs-bsde1}). Moreover, this solution is unique up to indistinguishability
for $Y$ and equivalence $\sim_{M}$ for $Z$.
\end{theorem}

\begin{proof}
It is clear that the solution $Y_{T}=\xi$ at time $T$. Then we construct the
solution for all $t$ using backward induction. Without loss of generality, we
only consider the following one-step RBSDE
\begin{equation}
Y_{t}=Y_{t+1}+f(t,Y_{t},Z_{t})+K_{t+1}-K_{t}-Z_{t}^{\ast}M_{t+1}.
\label{EU-bsde1}
\end{equation}

\noindent \textbf{(1) Existence.} We divide the proof into two steps. In
the first step, we construct a sequence of BSDEs and prove the convergence of
the corresponding solutions. We prove that the limit obtained in Step 1 is a
solution of (\ref{FS_BSDE}) in the second step.

\medskip
\noindent \textbf{Step 1.} Consider the following sequence of BSDEs
\begin{equation}
Y_{t}^{n}=Y_{t+1}+f(t,Y_{t}^{n},Z_{t}^{n})+n(Y_{t}^{n}-S_{t})^{-}-(Z_{t}%
^{n})^{\ast}M_{t+1},\quad n\in\mathbb{N} \label{EU-bsde2}%
\end{equation}
where $\mathbb{N}$ is the set of positive integers. Taking a
conditional expectation in (\ref{EU-bsde2}), we get
\begin{equation}
Y_{t}^{n}=E[Y_{t+1}|\mathcal{F}_{t}]+f(t,Y_{t}^{n},Z_{t}^{n})+n(Y_{t}%
^{n}-S_{t})^{-},\quad n\in\mathbb{N}. \label{EU-bsde-3}
\end{equation}
Hence,
\[
(Z_{t}^{n})^{*}M_{t+1}=Y_{t+1}-E[Y_{t+1}|\mathcal{F}_{t}].
\]

By the Martingale Representation Theorem (Theorem \ref{MRT}), there exists
a unique process $Z_{t}$, up to equivalence $\sim_{M}$, such that the above
equation is satisfied for an arbitrary $n$. Using this $Z_{t}$,
(\ref{EU-bsde-3}) can be rewritten as
\begin{equation}
Y_{t}^{n}=Y_{t+1}+f(t,Y_{t}^{n},Z_{t})+n(Y_{t}^{n}-S_{t})^{-}-Z_{t}^{\ast}M_{t+1},\quad n\in\mathcal{N}. \label{EU-bsde4}
\end{equation}

Let
\[
{ f_{n}(t,y,z)=f(t,y,z)+n(y-S_{t})^{-}}\text{.}%
\]
Then $y-f_{n}(t,y,z)$ is strictly increasing and continuous in
$y$. By theorem \ref{Comparison Theorem1}, (\ref{EU-bsde4}) has a unique
solution $(Y_{t}^{n},Z_{t})$. It is clear that
\begin{enumerate}[(i)]
\item $f_{n+1}(t,y,z)\geq f_{n}(t,y,z),$ \quad $\forall(y,z)\in \mathbb{R}\times \mathbb{R}^{m}$;
\item $f(t,Y_{t},Z_{t}^{n})-f(t,Y_{t},Z_{t}^{n+1})=0=\underset{i\in\mathcal{Q}_{t}}{\min}\{[Z_{t}^{n}-Z_{t}^{n+1}]^{\ast}(e_{i}-E[X_{t+1}])$, as $Z^n=Z^{n+1}=Z$;
\item since the map $y-f_{n}(t,y,z)$ is strictly increasing, we obtain that if
\[y_{1}-f_{n}(t,y_{1},z)\geq y_{2}-f_{n}(t,y_{2},z),\]
then $y_{1}\geq y_{2}$  $P$-a.s.
\end{enumerate}
Therefore, by the Comparison Theorem \ref{Comparison Theorem1}, we see that $Y_{t}^{n+1}\geq Y_{t}^{n}$ $P$-a.s. Thus we have the existence of a limit
\[Y_{t}^{n}\uparrow Y_{t}\quad P-a.s.\]

From (\ref{EU-bsde-3}), on the event $\{Y_{t}^{n}\geq S_{t}\}$, we have
$Y_{t}^{n}=E[Y_{t+1}|\mathcal{F}_{t}]+f(t,Y_{t}^{n},Z_{t})$. Because the map
$y\mapsto y-f(t,y,z)$ is strictly increasing, we deduce that $Y_{t}^{n}$ is essentially
bounded on $\{Y_{t}^{n}\geq S_{t}\}$. On the event $\{Y_{t}^{n}<S_{t}\}$,
\[Y_{t}^{n}-f(t,Y_{t}^{n},Z_{t})=E[Y_{t+1}|\mathcal{F}_{t}]+n(S_{t}-Y_{t}^{n}).\]
Since  $y-f(t,y,z)$ is strictly increasing and $n(S_{t}- Y_{t}^{n})\geq 0$ on $\{Y_{t}^{n}<S_{t}\}$, there
exists a lower bound for ${ Y_{t}^{n}}$ on this event.
Combining these bounds, from Fatou's Lemma we see that
\[E\mid Y_{t}\mid\leq\lim_{n\rightarrow\infty}E\mid Y_{t}^{n}\mid<\infty.\]

Define a process $K^n$ by $K^n_0=0$ and
\[
K_{t+1}^n-K_t^n=\Delta K_{t}^{n}=n(Y_{t}^{n}-S_{t})^{-}.
\]
 By (\ref{EU-bsde-3}), we have
\[
|\Delta K_{t}^{n+p}-\Delta K_{t}^{n}|\leq |f(t,Y_{t}^{n+p},Z_{t})-f(t,Y_{t}^{n},Z_{t})|+|Y_{t}^{n+p}-Y_{t}^{n}|, \quad \text{for all } p\in\mathbb{N}.\]

Since $f$ is continuous in $y$ and $Y_{t}^{n}\uparrow Y_{t}$ $P$-a.s., we obtain
\[
|\Delta K_{t}^{n+p}-\Delta K_{t}^{n}| \rightarrow 0,\quad \text{as }n\rightarrow \infty.\]
Consequently, there exist random variables $\Delta K_t$ such that $\Delta K_{t}^{n}\rightarrow\Delta K_{t}=K_{t+1}-K_t$ as $n\rightarrow\infty$.
Define the limiting process $K$ by
\[K_{0}=0\text{ and }K_{t}=\sum_{0\leq u<t}\Delta K_{u}.\]
Then as $n\rightarrow\infty$, (\ref{EU-bsde4}) becomes
\[Y_{t}=Y_{t+1}+f(t,Y_{t},Z_{t})+K_{t+1}-K_{t}-Z_{t}^{\ast}M_{t+1}.\]

\noindent\textbf{Step 2. } It is clear that the triple $(Y_{t},Z_{t},K_{t})$ obtained above
 satisfies (i) and (ii) of Definition \ref{FS_BSDE}. It remains to check (iii) and (iv).

First, note that $K_{t}$ is increasing, as $\Delta K_{t}$ is non-negative. As
\[(Y_{t}^{n}-S_{t})\Delta K_{t}^{n}=n(Y_{t}^{n}-S_{t})(Y_{t}^{n}-S_{t})^{-}=-n[(Y_{t}^{n}-S_{t})^{-}]^{2}\leq0,\]
we have that
\[(Y_{t}-S_{t})(K_{t+1}-K_{t})\leq0.\]

On the other hand, as $Y^n$ is increasing in $n$,
\[(Y_{t}^{n+1}-S_{t})^{-}\leq(Y_{t}^{n}-S_{t})^{-}.\]
By (\ref{EU-bsde-3}), we have
\[(Y_{t}^{n}-S_{t})^{-}=\frac{Y_{t}^{n}-E[Y_{t+1}|\mathcal{F}_{t}]-f(t,Y_{t}^{n},Z_{t})}{n}.\]
Then, as $n\rightarrow\infty$,
\[(Y_{t}^{n}-S_{t})^{-}\downarrow0\text{ and}\quad(Y_{t}-S_{t})^{-}=\lim_{n\rightarrow+\infty}(Y_{t}^{n}-S_{t})^{-}=0.\]
It follows that $Y_{t}\geq S_{t}$. Hence
\[(Y_{t}-S_{t})(K_{t+1}-K_{t})\geq0\quad P\text{-a.s.}\]
 Thus, we obtain $(Y_{t}-S_{t})(K_{t+1}-K_{t})=0$ $P$-a.s.

\medskip
\noindent \textbf{(2) Uniqueness.} Suppose that there exist two solutions
$(Y_{t},Z_{t},K_{t})$ and $(Y_{t}^{\prime},Z_{t}^{\prime},K_{t}^{\prime})$ of the
RBSDE (\ref{EU-bsde1}).  Without loss of
generality, suppose $Y_{t}>Y_{t}^{\prime }$ and $Y_{s}=Y_{s}^{\prime }$ for
all $s\in \{0,1,...,(t-1)\}$.
Then $Y_{t}>Y_{t}^{\prime}\geq S_{t}$. It follows that $K_{t+1}-K_{t}=0$ and (\ref{EU-bsde1}) can be simplified to
\[Y_{t}=Y_{t+1}+f(t,Y_{t},Z_{t})-Z_{t}^{\ast}M_{t+1}.\]
 On the other hand,
\[Y_{t}^{\prime}=Y_{t+1}+f(t,Y_{t}^{\prime},Z_{t}^{\prime})+K_{t+1}^{\prime}-K_{t}^{\prime}-Z_{t}^{\prime*}M_{t+1}.\]

By Theorem \ref{Comparison Theorem1}, we have $Y_{t}\leq Y_{t}^{\prime}$ $P$-a.s. This leads to
contradiction. Thus, we have $Y_{t}=Y_{t}^{\prime}$  $P$-a.s.

By Corollary \ref{Corollary Comparison Theorem}, we have $K_{t}=K_{t}^{\prime}$, $K_{t+1}=K_{t+1}^{\prime}$, and consequently
\[\begin{split}
Z_{t}^{\prime\ast}M_{t+1}&=  Y_{t+1}-E[Y_{t+1}|\mathcal{F}_{t}]+K_{t+1}^{\prime}-K_{t}^{\prime}-E[K_{t+1}^{\prime}-K_{t}^{\prime}|\mathcal{F}_{t}]\\
&=  Y_{t+1}-E[Y_{t+1}|\mathcal{F}_{t}]+K_{t+1}-K_{t}-E[K_{t+1}-K_{t}|\mathcal{F}_{t}]\\
&= Z_{t}^{\ast}M_{t+1}
\end{split}
\]
By Definition \ref{def:simM}, we have $Z\sim_{M}Z^{\prime}$.
\end{proof}


\subsection{Relation to optimal stopping problems}

We now show that the solution $(Y_{t})$ of the RBSDE
(\ref{fs-bsde1}) corresponds to the value of an optimal stopping problem.

\begin{proposition}\label{propositation fsbsde} Let $\{(Y_{t},Z_{t},K_{t})\}_{0\leq t\leq T}$ be a
solution of the RBSDE (\ref{fs-bsde1}). Then for each
$t\in\{0,1,...,T\}$,
\[
Y_{t}=\underset{\theta\in\mathcal{J}_{t}}{\sup}E\Big[\sum_{t\leq s<\theta
}f(s,Y_{s},Z_{s})+S_{\theta}1_{\{\theta<T\}}+\xi1_{\{\theta=T\}}\Big|\mathcal{F}_{t}\Big],
\]
where $\mathcal{J}$ is the set of all stopping times dominated by $T$ and
$\mathcal{J}_{t}\triangleq\{\theta\in\mathcal{J};t\leq\theta\leq T\}$.
\end{proposition}

\begin{proof}
For a given stopping time $\theta\in\mathcal{J}_{t}$, we have
\[Y_{t}=Y_{\theta}+\sum_{t\leq u<\theta}{f(u,Y_{u},Z_{u})}+K_{\theta}-K_{t}-\sum_{t\leq u<\theta}{Z_{u}^{\ast}M_{u+1}},\qquad0\leq t\leq T.
\]
Taking the conditional expectation,
\begin{align*}
Y_{t}  & =E[Y_{\theta}+\sum_{t\leq u<\theta}{f(u,Y_{u},Z_{u})}+K_{\theta}-K_{t}|\mathcal{F}_{t}]\\
&  \geq E[\sum_{t\leq u<\theta}{f(u,Y_{u},Z_{u})}+S_{\theta}1_{\{ \theta<T\}}+\xi1_{\{ \theta=T\}}|\mathcal{F}_{t}].
\end{align*}

In order to obtain the reversed inequality, we define
\[D_{t}=\begin{cases}
         T, & \text{if }Y_{u}>S_{u}\text{ for all } {t\leq u\leq T};\\
\inf\{u: t\leq u\leq T,  Y_{u}=S_{u}\}, & \text{otherwise.}%
        \end{cases}
\]

Note that $\sum_{0\leq t\leq T}(Y_{t}-S_{t})(K_{t+1}-K_{t})=0$
implies that $K_s= K_{s-1}$ for any $t+1\leq s\leq D_{t}$, and so
\[
K_{D_{t}}-K_{t}=\sum_{t+1\leq s\leq D_{t}}(K_{s}-K_{s-1})=0,\quad0\leq t\leq T.
\]
From this, we see that
\begin{align*}
Y_{t} & =E\Big[Y_{D_{t}}+\sum_{t\leq u<D_{t}}{f(u,Y_{u},Z_{u})}+K_{D_{t}}-K_{t}\Big|\mathcal{F}_{t}\Big]\\
&  =E\Big[Y_{D_{t}}+\sum_{t\leq u<D_{t}}{f(u,Y_{u},Z_{u})}\Big|\mathcal{F}_{t}\Big]\\
&  \leq\underset{\theta\in\mathcal{J}_{t}}{\sup}E\Big[\sum_{t\leq u<\theta}f(u,Y_{u},Z_{u})+S_{\theta}1_{\{\theta<T\}}+\xi1_{\{\theta=T\}}
\Big|\mathcal{F}_{t}\Big]
\end{align*}
which completes the proof.
\end{proof}

\begin{example}
Set $L_{t}=\sum_{t\leq u<T}Z_{u}^{\ast}M_{u+1}$. Consider the special case
$f=C$, $S_{T}=\xi\geq0$ where $C$ is a constant. If $(Y,Z, K)$ is a solution, then
\[Y_{0} =E[\xi+CT+K_{T}]  =E\Big[\xi+\underset{0\leq t\leq T}{\sup}(S_{t}+L_{t}-C(T-t)-\xi)^{+}\Big].\]
Since $S_{T}=\xi$, it is easy to check that
\[Y_{0}=\sup_{\theta\in\mathcal{J}_{0}}E[S_{\theta}+C\theta]=E\Big[\underset{0\leq t\leq T}{\sup}(S_{t}+L_{t}+Ct)\Big].
\]
 When $C=0$, we have
\[Y_{0}=\sup_{\theta\in\mathcal{J}_{0}}E[S_{\theta}]=E\Big[\underset{0\leq t\leq
T}{\sup}(S_{t}+L_{t})\Big].\]
\end{example}

By Proposition \ref{propositation fsbsde}, the solution of the
RBSDE in which $f$ is a given stochastic process is the value function of an
optimal stopping problem. In the following, we shall particularly investigate the
cases where $f(t,y,z)$ is a linear function or concave (convex) function. In
the latter case, the solution $\{Y_{t}\}_{0\leq t\leq T}$ is
shown to be the value function of a mixed optimal stopping--optimal stochastic
control problem. Note that El. Karoui et al. \cite{r10} studied similar
problems for reflected backward stochastic differential equations in continuous time.

Without loss of generality, we consider the one-step RBSDE
in our framework, i.e.
\begin{equation}
Y_{t}=Y_{t+1}+f(t,Y_{t},Z_{t})-Z_{t}^{\ast}M_{t+1}+K_{t+1}-K_{t},\qquad 0\leq t<T.
\label{FS-bsde1}
\end{equation}
%
%
%

By Proposition \ref{propositation fsbsde}, we have the following
results.

\begin{proposition}
\label{proposiztation solution} Consider the RBSDE (\ref{fs-bsde1}) with
coefficient $f=\alpha_{t},$ where $\{\alpha_{t}\}_{0\leq t\leq T}$ is a given
adapted process and takes values in $R$, and $S$ is a given adapted boundary process. Then the unique solution $(Y,Z,K)$
satisfies
\[Y_{t}=\sup_{\theta\in\mathcal{J}_{t}}E\Big[\sum_{t\leq s<\theta}\alpha
_{s}+S_{\theta}I_{\{\theta<T\}}+\xi I_{\{\theta=T\}}\Big|\mathcal{F}_{t}\Big].
\]
Moreover, if we consider equation (\ref{FS-bsde1}), then we have
\begin{equation}
Y_{t}=S_{t}\vee(\alpha_{t}+E[Y_{t+1}|\mathcal{F}_{t}%
]).\label{solution of one step}%
\end{equation}

\end{proposition}
\begin{proof}
 In fact, we can directly obtain (\ref{solution of one step}) from the definition of an RBSDE solution (Definition
\ref{FS_BSDE}). Denote
\[
\rho_{t}=\alpha_{t}+E[Y_{t+1}|\mathcal{F}_{t}].
\]
Taking the conditional expectation for (\ref{FS-bsde1}), we have
\[
Y_{t}=\rho_{t}+E[K_{t+1}-K_{t}|\mathcal{F}_{t}].
\]

There are then two cases:
\begin{enumerate}[(i)]
\item $\rho_{t}\geq S_{t}$. Then $Y_{t}-S_{t}\geq E[K_{t+1}-K_{t}|\mathcal{F}_{t}]\geq 0$, since $K$ is an increasing process. However, by condition (iv) of Definition \ref{FS_BSDE}, it follows that $K_{t+1}-K_{t}=0$, so $Y_{t}=\rho_{t}$.
\item $\rho_{t}<S_{t}$. It follows that $K_{t+1}-K_{t}>0$. By condition (iv) of Definition \ref{FS_BSDE}
we have $Y_{t}=S_{t}$.
\end{enumerate}
{ To sum up, $Y_{t}=S_{t}\vee\rho_{t}$, i.e. $Y_{t}=S_{t}%
\vee(\alpha_{t}+E[Y_{t+1}|\mathcal{F}_{t}])$. }
\end{proof}

To neatly consider linear RBSDEs, we need the following definition.
\begin{definition}\label{defQvector}
 Recall from (\ref{Qtset}) that $\mathcal{Q}_t$ defines the set of possible jumps of $X$ at time $t$. We shall say that $\gamma$ is a $Q$-vector process if it is an  adapted process in $L^1(\mathbb{R}^m)$ which satisfies $\sum_j \langle \gamma_t, e_j\rangle = 0$ and $\langle \gamma_t, e_i\rangle =0$ for all $e_i\not\in \mathcal{Q}_t$. We write $\mathbb{R}^m_{Q,t}$ for the space of $Q$-vectors at time $t$, and note this is a subspace of $\mathbb{R}^m$.
\end{definition}
\begin{lemma}
For $\gamma$ a $Q$-vector process, the function $f(\omega,t, z) = \langle \gamma_t, z\rangle $ satisfies Theorem \ref{thm:BSDEexist} condition (i). Furthermore, the solution process $Z$ can be taken to lie in $\mathbb{R}^m_{Q,t}$ without loss of generality, and is unique in this space.
\end{lemma}
\begin{proof}
 Simply note that $Z \sim_M Z'$ if and only if, for every $t$, $Z_t$ and $Z'_t$ differ at most by a constant and by the values of $\langle Z_t, e_i\rangle$ for $e_i\notin\mathcal{Q}_t$, both of which are in the kernel of the linear map $\langle \gamma_t, \cdot \rangle$.
\end{proof}

\begin{proposition}\label{proposiztation solution2}
Let $\{\alpha_{t},\beta_{t},\gamma_{t}\}_{0\leq t\leq T}$ be adapted processes
taking values in $\mathbb{R}\times\lbrack0,1)\times \mathbb{R}^{m}$, and let $\gamma$ be a $Q$-vector process. Let $S$ be a given adapted boundary process. Consider the
RBSDE (\ref{FS-bsde1}) with
\[
f(t,y,z)=\alpha_{t}+\beta_{t}y+\langle\gamma_{t},z\rangle.
\]
Then the solution $(Y,Z,K)$ satisfies
\begin{equation}
Y_{t}=S_{t}\vee(\alpha_{t}+\beta_{t}Y_{t}+\langle\gamma_{t},z\rangle+E[Y_{t+1}|\mathcal{F}_{t}]). \label{FS-bsde2}%
\end{equation}
\end{proposition}

\begin{proof}
It is easy to check that Theorem \ref{Existence Uniqueness} applies, so the solution $(Y,Z,K)$ exists and is unique. By Proposition
\ref{proposiztation solution} applied with $f(t,Y_t, Z_t)$ as the fixed term, we obtain (\ref{FS-bsde2}).
\end{proof}

Note that this also yields a simple method of calculating solutions. We have that $Z_{t}$ satisfies $Z_{t}^{\ast}M_{t+1}=Y_{t+1}-E[Y_{t+1}|\mathcal{F}_{t}]$ from
the proof of Theorem \ref{Existence Uniqueness}, and so we obtain
\[
\hat Y_{t}=\frac{1}{1-\beta_{t}}\big(\alpha_{t}+\langle\gamma_{t},Z_t\rangle+E[Y_{t+1}|\mathcal{F}_{t}]\big).
\]
For each $(\omega, t)$, if $\hat Y_{t}>S_{t}$, then $Y_t=\hat Y_t$ is as desired; otherwise, $Y_{t}=S_{t}$.

\begin{remark}
If $\beta_{t}=1$, then
\[y-f(t,y,z)=-\alpha_{t}-\langle\gamma_{t},z\rangle\]
which violates assumption (ii) of Theorem \ref{Existence Uniqueness}, as the right hand side is independent of $y$. Thus, we
can not guarantee that there exists a unique solution of the linear RBSDE (and typically, no solution will exist).
\end{remark}

\subsubsection{Concave coefficients}
We now suppose that for each fixed $(\omega,t)$, the driver $f(t,y,z)$ is a
concave function of $(y,z)$. For each $(\omega,t,\beta,\gamma)\in\Omega\times\{0,1,...,T\} \times \mathbb{R}\times \mathbb{R}^{m}$, define the conjugate
function $F(t,\beta,\gamma)$ as follows:
\[\begin{split}
F(\omega,t,\beta,\gamma)&=\sup_{(y,z)\in\mathbb{R}\times \mathbb{R}^m_{Q,t}}(f(t,y,z)-\beta y-\langle\gamma,z\rangle)\\
D_{t}^{F}(\omega)&=\{(\beta,\gamma)\in \mathbb{R}\times \mathbb{R}^{m}_{Q,t};F(\omega,t,\beta,\gamma)<\infty\}.
\end{split}
\]
It follows that
\[
f(t,y,z)=\inf_{(\beta,\gamma)\in D_{t}^{F}}\{F(t,\beta,\gamma)+\beta y+\langle \gamma,z\rangle\},\]
the infimum is achieved at $(\beta^{\prime},\gamma^{\prime})\in D_{t}^{F}$ and the set $D_{t}^{F}$ is a.s. bounded (refer to \cite{EPQ}).  If the function $f$ also satisfies conditions (iii) and (iv) of the Comparison Theorem (Theorem \ref{Comparison Theorem1}), we see that the infimum is attained in the smaller set
\begin{equation}\label{eq:Cdef}
 \begin{split}
C_{t}^{F}(\omega)&=\big\{(\beta,\gamma)\in D_t^F: |\beta_t|<1,\\
    &\qquad\qquad \langle \gamma_t, z\rangle \geq \langle e_i - E[X_{t+1}|\mathcal{F}_t], z\rangle\text{ for all }z\in\mathbb{R}^m_{Q,t}, e_i\in \mathcal{Q}_t\big\}.
  \end{split}
\end{equation}

Denote the solution of RBSDE with coefficient
\[f^{\beta,\gamma}(t,y,z)=F(t,\beta_{t},\gamma_{t})+\beta_{t}y+\langle\gamma_{t},z\rangle
\]
 by $\{(Y_{t}^{\beta,\gamma},Z_{t}^{\beta,\gamma},K_{t}^{\beta,\gamma
})\}_{0\leq t\leq T}$ (resp. $\{(Y_{t},Z_{t},K_{t})\}_{0\leq t\leq T}$ for the RBSDE with coefficient $f(t,y,z)$). Then, $P$-a.s. for all $t$, we have
\begin{align*}
f(t,Y_{t},Z_{t})  &  =F(t,\beta^{\prime},\gamma^{\prime})+\beta^{\prime}Y_{t}+\langle\gamma^{\prime},Z_{t}\rangle\\
(Y_{t},Z_{t},K_{t})  &  =(Y_{t}^{\beta^{\prime},\gamma^{\prime}},Z_{t}^{\beta^{\prime},\gamma^{\prime}},K_{t}^{\beta^{\prime},\gamma^{\prime}})\quad
\end{align*}
and so $Y_{t}=Y_{t}^{\beta^{\prime},\gamma^{\prime}}$ can be
interpreted as the value functions of an optimization problem.

\begin{theorem}\label{convexdriverRBSDE}
For each $(\beta_{t},\gamma_{t})\in C_{t}^{F}$ with $|\beta_{t}|<1$ $P$-a.s., we have
\begin{align*}
Y_{t}^{\beta,\gamma} &  =S_{t}\vee(F(t,\beta_{t},\gamma_{t})+\beta_{t} Y_{t}^{\beta,\gamma}+\langle \gamma_{t},Z_{t}\rangle +E[Y_{t+1}|\mathcal{F}_{t}]);\\
Y_{t} &  =S_{t}\vee(F(t,\beta_{t}^{\prime},\gamma_{t}^{\prime})+\beta_{t}^{\prime} Y_{t}+\langle \gamma_{t}^{\prime},Z_{t}\rangle +E[Y_{t+1}|\mathcal{F}_{t}]).
\end{align*}
Moreover,
\begin{align*}
Y_{t} &  =\inf_{(\beta,\gamma)\in C_{t}^{F}}Y_{t}^{\beta,\gamma}\\
&  =\inf_{(\beta,\gamma)\in C_{t}^{F}}\big(S_{t}\vee(F(t,\beta_{t},\gamma_{t})+\beta_{t}Y_{t}^{\beta,\gamma}+\langle \gamma_{t},Z_{t}\rangle+E[Y_{t+1}|\mathcal{F}_{t}])\big)\\
&  =S_{t}\vee\inf_{(\beta,\gamma)\in C_{t}^{F}}(F(t,\beta_{t},\gamma
_{t})+\beta_{t}Y_{t}^{\beta,\gamma}+\langle\gamma_{t},Z_{t}\rangle+E[Y_{t+1}%
|\mathcal{F}_{t}]).
\end{align*}

In other words, $Y_{t}$ is the value function of a minimax
control problem, and the triple ($\beta^{\prime},\gamma^{\prime},D_{t}$),
where $D_{t} =\inf\{s:t\leq s\leq T, Y_{s}=S_{s}\}$ is optimal.
\end{theorem}

\begin{proof}
The first statement is apparent from Proposition \ref{proposiztation solution2}. By the Comparison Theorem \ref{Comparison Theorem}, we have
\[
Y_{t}\leq Y_{t}^{\beta,\gamma},\quad\text{for all }(\beta,\gamma)\in C_{t}^{F}.
\]
On the other hand,
\[
Y_{t}=Y_{t}^{\beta^{\prime},\gamma^{\prime}}\geq\inf_{(\beta,\gamma)\in
D_{t}^{F}}Y_{t}^{\beta,\gamma},
\]
which immediately leads to
\[
Y_{t}=\inf_{(\beta,\gamma)\in C_{t}^{F}}Y_{t}^{\beta,\gamma}.
\]
Finally, it is easy to see that the operators $\inf$ and $\vee$ can be exchanged.
\end{proof}

\begin{remark}\label{supstopping}
If $f$ is a convex function of $(y,z)$, by essentially the same argument, we have a similar representation of the form
\[Y_{t}  =S_{t}\vee\sup_{(\beta,\gamma)\in C_{t}^{(-F)}}(F(t,\beta_{t},\gamma
_{t})+\beta_{t}Y_{t}^{\beta,\gamma}+\langle\gamma_{t},Z_{t}\rangle+E[Y_{t+1}|\mathcal{F}_{t}])
\]
where $F(t, \beta, \gamma) = \inf_{(y,z)\in \mathbb{R}\times \mathbb{R}^m_{Q,t}} (f(t,y,z)-\beta y-\langle\gamma,z\rangle)$, and $C^{(-F)}_t$ is as defined in (\ref{eq:Cdef}).
\end{remark}

\section{Optimal stopping under $g$-expectation}

In order to studyoptimal stopping problems under
$g$-expectation, we first study the $g$-expectation theory on spaces related to
discrete time, finite state processes. Then we give the connection between
multiple prior martingales and $g$-martingales. Last, we show that some
optimal stopping problems with multiple priors can be solved by computing the
corresponding RBSDEs.

\subsection{$g$-expectation}

Peng \cite{r18} and \cite{r19} introduced the notions of
$g$-expectation and conditional $g$-expectation as well as $g$-martingale via
backward stochastic differential equations. The aim of this section is to study the
$g$-expectation theory in our framework.

\begin{definition}
 We say a driver $f$ satisfying the conditions of Theorem \ref{Existence Uniqueness} is $\emph{normalised}$ if
\[f(\omega,t, y,0) = 0\quad P\text{-a.s. for all }t\in\{0,1,...,T\}, y\in \mathbb{R}.\]
\end{definition}

\begin{definition}\label{def:gexp}
 A system of operators
\[G(\cdot|\mathcal{F}_t) : L^1(\mathcal{F}_T; \mathbb{R}) \to L^1(\mathcal{F}_t; \mathbb{R})\]
is called a filtration consistent nonlinear expectation if it satisfies, for all $\xi\in L^1(\mathcal{F}_T; \mathbb{R})$, all $t$,
\begin{enumerate}[(i)]
 \item If $\xi\geq \xi'$ $P$-a.s. then $G(\xi|\mathcal{F}_t) \geq G(\xi'|\mathcal{F}_t)$.
 \item For any $\mathcal{F}_t$-measurable $\xi$, we have $G(\xi|\mathcal{F}_t) = \xi$.
 \item We have the tower property $G(G(\xi|\mathcal{F}_t)|\mathcal{F}_s) = G(\xi|\mathcal{F}_s)$ for  all $s<t$.
\item For any $A\in \mathcal{F}_t$, we have $I_A G(\xi|\mathcal{F}_t) = G(I_A \xi|\mathcal{F}_t)$.
\end{enumerate}
It is said to be translation invariant if
\begin{enumerate}[(i)]
 \setcounter{enumi}{4}
\item For any $q\in L^1(\mathcal{F}_t, \mathbb{R})$ we have $G(\xi+q|\mathcal{F}_t) = G(\xi|\mathcal{F}_t) +q$.
\end{enumerate}
\end{definition}

We have the following representation theorem from \cite{r8}, Theorem 7.

\begin{theorem} \label{gexprepresentation}The following are equivalent.

\begin{enumerate}[(i)]
 \item $G$ is a filtration consistent, translation invariant nonlinear expectation
 \item $Y_t=G(\xi|\mathcal{F}_t)$ is the solution to a BSDE with coefficient $f:\Omega\times \{0,1,...,T\}\times \mathbb{R}^m \to \mathbb{R}$, and
$f$ satisfies the conditions of Theorem \ref{Existence Uniqueness} (so BSDE solutions exist) and conditions (iii) and (iv) of Theorem  \ref{Comparison Theorem1} (so the comparison theorem holds) and is normalised.
\end{enumerate}
Furthermore, the function $f$ is unique, and can be obtained from the relation $f(\omega, t, z) = G(z^* M_{t+1}|\mathcal{F}_t)$.
\end{theorem}

From \cite{r18} and \cite{r19}, an operator $G$ defined by a BSDE solution of this type is called a $g$-expectation. This result therefore shows that $g$-expectations and filtration consistent, translation invariant nonlinear expectations coincide. For the sake of brevity, we will therefore use the term $g$-expectation.

Now we study the Doob--Meyer decomposition and optional sampling
theorem under $g$-expectation.
\begin{definition}
 A process $\{X_t\}$ will be called a $g$-supermartingale  if $X_t\in L^1(\mathcal{F}_t; \mathbb{R})$ for all $t$ and  $X_s\geq G(X_t|\mathcal{F}_t)$ for all $s\leq t$. In a similar way, we define submartingales and martingales.
\end{definition}
Recall that, in this setting, a process $K$ is predictable if $K_{t+1}$ is $\mathcal{F}_t$-measurable for all $t$.
\begin{theorem}[Doob--Meyer Decomposition]
 Let $X$ be a $g$-supermartingale (resp. $g$-submartingale). Then there exists a unique predictable increasing (resp. decreasing) process $K$ such that $X+K$ is a $g$-martingale and $K_0=0$.
\end{theorem}
\begin{proof}
 Let $K_0=0$ and define $K$ recursively by
\[K_{t+1} = K_{t} + G(X_{t+1}|\mathcal{F}_t) - X_t.\]
Then simple calculation verifies the result.
\end{proof}

In the following, we study the optional sampling theorem for
$g$-super and $g$-sub-martingales in our framework.

\begin{theorem}
 Let $X$ be a $g$-supermartingale. Then for any stopping times $\sigma, \tau$ with $0\leq \sigma\leq \tau \leq T$, we have
\[X_\sigma \geq G(X_\tau|\mathcal{F}_\sigma).\]
Similarly for $g$-submartingales.
\end{theorem}
\begin{proof}
 From the Doob--Meyer decomposition, we know that there is an increasing process $K$ such that $X+K$ is a $g$-martingale. Hence, by the recursive nature of BSDEs and normalisation
\[\begin{split}X_\sigma + K_\sigma &= X_\tau + K_\tau + \sum_{\sigma\leq t<\tau} f(t, Z_t) - \sum_{\sigma\leq t<\tau} Z_t^*M_{t+1}\\
   &= X_\tau + K_\tau + \sum_{\sigma\leq t<T} f(t, Z_t) - \sum_{\sigma\leq t<T} Z_t^*M_{t+1}.
  \end{split}
\]
 Rearranging the terms, we have
\[X_\sigma  = X_\tau  + \sum_{\sigma\leq t<\tau} f^K(t, Z_t) - \sum_{\sigma\leq t<\tau} Z_t^*M_{t+1}\]
where
\[f^K(t, Z_t)= f(t, Z_t) + K_{t+1}-K_t.\]
As $K$ is an increasing predictable process, $f^K$ satisfies the requirements of Theorem \ref{Comparison Theorem}, and $f^K(t,Z_t) \geq f(t, Z_t)$. Therefore, $X_\sigma\geq G(X_\tau|\mathcal{F}_\sigma)=Y_\sigma$, where $Y_\sigma$ solves the BSDE
\[Y_\sigma  = Y_\tau  + \sum_{\sigma\leq t<\tau} f(t, Z_t) - \sum_{\sigma\leq t<\tau} Z_t^*M_{t+1}.\]
The corresponding argument when $X$ is a submartingale, and so $K$ is a decreasing process, also holds.
\end{proof}

\subsection{Multiple prior martingales and $g$-Martingales}

Riedel \cite{r22} developed a theory of optimal stopping under
multiple priors. He defined the a process $\{U_t\}$ to be a `multiple prior martingale' if it
satisfies
\[
U_{t}=\inf_{p\in\Lambda}E^{p}[U_{t+1}|\mathcal{F}_{t}],
\]
 where $\Lambda$ is a set of time-consistent measures, as defined in the following.

\begin{definition}\label{defnTimeconsistentMeasures}
 A family $\Lambda$ of probability measures will be called `time-consistent' if, for any $Q, Q'\in \Lambda$, any $A\in \mathcal{F}_t$, we have $Q''\in \Lambda$, where
\[Q''(B) = E_Q[ I_A E_{Q'}[ I_B|\mathcal{F}_t] + I_{A^c} I_B].\]
\end{definition}
See the $m$-stability of \cite{Delbaen}, Proposition 3.6 in \cite{NutzSoner}, Theorem 2.2 in \cite{r3} and Definition 13 in \cite{CohenAggregation} for discussion of this and related concepts. In particular, we have the following result, which is proven in each of these references in varying degrees of generality (any of which is sufficient for our setting here).
\begin{lemma}\label{timeconsistentequiv}
The family of operators $G(\cdot|\mathcal{F}_t)=\inf_{p\in\Lambda}E^{p}[U_{t+1}|\mathcal{F}_{t}]$ is filtration consistent (in particular satisfies (ii-iv) of Definition \ref{def:gexp}) if and only if $\Lambda$ is a time-consistent family of measures.
\end{lemma}

In this subsection, we will study special cases of multiple prior martingales using the theory of (R)BSDEs.

We denote by $\mathcal{Q}$ the set of all probability measures
$Q\sim P$. For any $Q\in\mathcal{Q}$, set
\[W_{t}:=E^P\Big[\frac{dQ}{dP}\Big|\mathcal{F}_{t}\Big].\]
Then $W_{t}$ is a martingale and $W_0=1$. By Theorem \ref{MRT}, there exists an adapted process $z$ such that
\[W_{t}=1+\sum_{0\leq s<t}z_{s}^{\ast}M_{s+1}.\]

Let
\[\begin{split}
\psi_t&=\var(X_{t+1}|\mathcal{F}_{t})= E[X_{t+1}X_{t+1}^{\ast}|\mathcal{F}
_{t}]-E[X_{t+1}|\mathcal{F}_{t}]E[X_{t+1}^{\ast}|\mathcal{F}_{t}]\\
&=E[M_{t+1} M_{t+1}^*|\mathcal{F}_t],\end{split}
\]
so $\psi$ is a symmetric positive semidefinite matrix with null space orthogonal to the space of $Q$-vectors. We note that $\psi$ also appeared in Definition \ref{def:simM}. As the $z$ process in the martingale representation theorem is only defined up to equivalence $\sim_M$, writing
\[\theta_t = \frac{1}{W_t} \psi_t z_t,\]
and $\psi^+$ for the Moore--Penrose pseudoinverse of $\psi$, we have $\{W_t \psi_t^+ \theta_t\}_{0\leq t\leq T} \sim_M z$. Without loss of generality, we can take $\theta_t$ to be a $Q$-vector.

Therefore, we can write
\[W_{t}=\prod_{0\leq s<t}(1+\theta_{s}^*\psi_t^+ M_{s+1})\]
and
\begin{equation}
\frac{dQ}{dP}=W_{T}=\prod_{0\leq s<T}(1+\theta_{s}^*\psi_t^+M_{s+1}). \label{kn-1}
\end{equation}
Thus, for any $Q\in\mathcal{Q}$, $\frac{dQ}{dP}$ can be generated by
$(\theta_{t})$ through (\ref{kn-1}). The probability measure generated by
$(\theta_{t})$ is denoted by $Q^{\theta}$. It is classical that this is a probability measure if and only if $\theta_{s}^*\psi_s^+M_{s+1} >-1$ a.s.

Basic calculation yields that
\[E_{Q^\theta}[X_{t+1}|\mathcal{F}_t] = E[(1+\theta_t^*\psi_t^+ M_{t+1})X_{t+1}] = E[X_{t+1}|\mathcal{F}_t]+\psi_t^+\psi_t\cdot \theta_t.\]
In particular, for $\theta_t$ a $Q$-vector (and so orthogonal to the null space of $\psi$),
\[E_{Q^{\theta}}[M_{t+1}|\mathcal{F}_t] = \theta_t.\]
\begin{proposition}\label{BSDEinfexp}
If $\Lambda$ is a time consistent family of measures, by Lemma \ref{timeconsistentequiv} and Theorem \ref{gexprepresentation}, the process $Y_t = G(\xi|\mathcal{F}_t)$ solves a BSDE with terminal value $Y_T=\xi$ and driver
\[f(\omega,t,z) = \inf_{\{\theta:Q^{\theta}\in\Lambda\}}\{z^*\theta_t\}.\]
 \end{proposition}
Conversely, we can verify that $Q^\theta$ defined in this way is a probability measure (absolutely continuous with respect to $P$) provided that
\begin{itemize}
\item $\theta_t$ is a.s. a $Q$-vector for all $t$ and
\item $0 \leq E[X_{t+1}|\mathcal{F}_t] + \theta_t \leq 1$ a.s., the inequality being taken componentwise.
\end{itemize}
The measure $Q^\theta$ is equivalent to $P$ if and only if the inequality is strict in all components where $\theta_t\neq 0$.

\subsubsection{$\kappa$-ignorance model}
We now consider a concrete example, inspired by the $\kappa$-ignorance model in \cite{r3}. Suppose there exists a nonnegative process $\kappa_t$ such that
\[0 \leq E[X_{t+1}|\mathcal{F}_t] +\theta_t \leq 1\]
for all $Q$-vectors $\theta_t$ with $\|\theta_t\|_M \leq \kappa_t$. Consider the associated set of probability measures
\begin{equation}
\mathcal{B}=\Big\{Q^{\theta}: \|\theta_t\|_M \leq \kappa_t\, P\text{-a.s. for all }t\}.
\label{kn-2}%
\end{equation}

\begin{lemma}
$\mathcal{B}$ is a time consistent family of measures.
\end{lemma}
\begin{proof}
If $Q^\theta,Q^{\theta'} \in \mathcal{B}$, then for any time $t$, any $A\in \mathcal{F}_t$, the measure defined by
\[Q''(B) = E_{Q^\theta}[I_AE_{Q^{\theta'}}[I_B|\mathcal{F}_t]+I_{A^c} I_B]\]
will have representation $Q'' = Q^{\theta''}$, where
\[\theta''_s = \theta_{s\wedge t} + ((\theta_s-\theta_t)I_{A^c} + (\theta'_s - \theta'_t)I_{A})I_{s>t}.\]
Therefore, as $\theta''$ will also satisfy $\|\theta''_s\|_M\leq \kappa_s$,  the measure $Q''$ is also in $\mathcal{B}$.
\end{proof}

\begin{definition}
\label{min-1}Suppose $\xi\in L^{1}($$\mathcal{F}_{T};R)$. Let
\[
G(\xi|\mathcal{F}_{t})=\inf_{Q\in\mathcal{B}}\{E_{Q}[\xi|\mathcal{F}_{t}]\},\qquad 0\leq t\leq T.
\]
Then we call $G(\xi|\mathcal{F}_{t})$ the minimal conditional expectation of $\xi$ about $\mathcal{B}$.
Similarly, we can define the corresponding maximal conditional expectation.
\end{definition}

As $\mathcal{B}$ is a time consistent family, by Lemma \ref{timeconsistentequiv} and Theorem \ref{gexprepresentation}, we see the following relation.

\begin{theorem}
\label{gmax-1}For any $\xi\in L^{1}(\mathcal{F}_{T};R)$,  $Y_s = G(\xi|\mathcal{F}_s)$ is the solution to the BSDE
\[
Y_t = \xi - \sum_{t\leq s < T} \kappa_s \|Z_s\|_M -  \sum_{t\leq s < T} Z_s^* M_{s+1}
\]
\end{theorem}
\begin{proof}
We know from Theorem \ref{gexprepresentation} that $Y_s = G(\xi|\mathcal{F}_s)$ solves a BSDE with driver $f(\omega, t, z ) = G(z^* M_{t+1}|\mathcal{F}_t)$. Hence
\[f(\omega, t, z )= \inf_{Q\in\mathcal{B}}\{E_{Q}[z^* M_{t+1}|\mathcal{F}_{t}]\} = \inf_{\theta: \|\theta\|_M\leq \kappa_t}\{z^*\theta\}.\]
By the Cauchy--Schwartz inequality, this infimum is realised at $\theta \sim_M -\kappa\psi z$, where we have $f(\omega, t, z )=- \kappa_t \|z\|_M$.
\end{proof}
\subsubsection{Scenario perturbation model}
An alternative similar model is where a collection of perturbation vector processes $\{\pi^i\}_{i=1}^n$ are given, each of which takes values in the probability vectors in $\mathbb{R}^m$. We assume these are absolutely continuous with respect to $\pi^0_t:=E[X_{t+1}|\mathcal{F}_t]$, in the sense that if a component of $\pi^0_t$ is zero, then so is the corresponding component of $\pi^i_t$. These vectors can be thought of as `scenarios', or mixtures of scenarios, which with we will stress-test our outcome, and correspond to measures where $X_{t+1}$ has $\mathcal{F}_t$ conditional expectation $\pi^i_t$.

For a given  parameter $\kappa\leq 1$, we define a scenario perturbation measure to be a measure $Q^\theta$ where
\[\theta_t = \lambda_t (\pi^{\mathfrak{i}(t)}_t - \pi^0_t)\]
for some adapted processes $\lambda$ and $\mathfrak{i}$ with $\lambda_t\leq\kappa$ and $\mathfrak{i}(t)\in\{0,\ldots, n\}$. Again, one can verify that the associated family of measures is time consistent, and the corresponding minimal conditional expectations are given by the BSDE solutions
\[G(\xi|\mathcal{F}_t)=Y_t = \xi - \sum_{t\leq s < T} \kappa\min_i\{Z_s^*(\pi^{i}_s-\pi^0_s)\} -  \sum_{t\leq s < T} Z_s^* M_{s+1}\]

\subsection{Nonlinear expectations and optimal stopping}

Riedel \cite{r22} considered the optimal stopping problem under
ambiguity as follows:
\[
\text{maximize }\inf_{Q\in\Lambda}E^{Q}[U_{\tau}]\quad\text{over all
stopping times}\quad\tau\leq T
\]
for a finite horizon $T<\infty$, where $\Lambda$ is a time-consistent set of priors and
$(U_{t})_{t\in\mathcal{N}}$ is an essentially bounded and adapted process.

To solve the above problem, Riedel \cite{r22} introduced the
multiple prior Snell envelope $\bar{U}$ defined by $\bar{U}_{T}=U_{T}$ and
\[
\bar{U}_{t}=\max\{U_{t},\inf_{Q\in\mathcal{B}}E^{Q}[\bar{U}_{t+1}%
|\mathcal{F}_{t}]\},\;t\in\{0,1,...,T-1\}.
\]

We now study the relation between the multiple prior Snell envelope and RBSDEs. For a given time consistent family of measures $\Lambda$, let
\[\Theta= \{\theta: Q^\theta\in\Lambda\}\cap\{Q\text{-vectors}\}.\]
Consider the following RBSDE:
\begin{equation}
\begin{cases}
Y_{t}=Y_{t+1}+\inf_{\theta\in\Theta}\{Z_t^*\theta_t\}-Z_{t}^{\ast}M_{t+1}+K_{t+1}-K_{t}\\
Y_{T}=U_{T},\\
Y_{t}\geq U_{t},\\
(Y_{t}-U_{t})(K_{t+1}-K_{t})=0
\end{cases}
\label{kn-7}
\end{equation}

By Theorem \ref{Existence Uniqueness}, provided the infimum is almost surely finite (which is guaranteed by the fact that $\theta$ generates a measure), (\ref{kn-7}) has a unique solution $(Y_{t},Z_{t},K_{t})$. }

\begin{theorem}
\label{mutip copy(1)} Suppose $U_{T}\in L^{1}(\mathcal{F}_{T};R)$. Then the solution $Y_{t}$ of (\ref{kn-7}) is the multiple prior Snell envelope of $U$ with multiple
prior set $\Lambda$, that is, $Y=\bar U$.
\end{theorem}

\begin{proof}
By Theorem \ref{gexprepresentation}, we know that for any $\xi\in L^1(\mathcal{F}_{t+1};\mathbb{R})$
\[
\inf_{Q\in\Lambda}E^Q[\xi|\mathcal{F}_t] = \inf_{\theta\in\Theta}\{z^*\theta_t\}+E[\xi|\mathcal{F}_{t}]
\]
where $z^{\ast}M_{t+1}=\xi-E[\xi|\mathcal{F}_{t}]$. For $(Y,Z,K)$ the solution of (\ref{kn-7}), as $f(t, z) = \inf_{\theta\in\Theta}\{z^*\theta_t\}$ is concave, we know from Theorem \ref{convexdriverRBSDE} that
\[Y_t= U_t \vee \inf_{\theta\in\Theta}\{z^*\theta_t+E[\xi|\mathcal{F}_{t}]\} = \bar U_t\]
as desired.
\end{proof}

By the above theorem and Proposition \ref{propositation fsbsde}
as well as other properties of RBSDEs, we can deduce the following useful
results:

\begin{enumerate}[(i)]
 \item $\bar{U}$ is the smallest $g$-supermartingale (for the driver $f(t, z) = \inf_{\theta\in\Theta}\{z^*\theta_t\}$) which dominates $U$;
\item $\bar{U}$ is the value process of the following
optimal stopping problem under ambiguity, i.e.
\[
\bar{U}_{t}=\sup_{\tau\in\mathcal{J}_{t}}\inf_{P\in\mathcal{B}}E^{P}[U_{\tau}|\mathcal{F}_{t}];
\]
\item an optimal stopping rule can be given by
\[
\tau^{*}=\inf\{t\geq0: \bar{U}_{t}=U_{t}\}.
\]
\end{enumerate}

Now we reconsider a simple example which was discussed in \cite{r16} and
\cite{r5}.

\begin{example}
\label{Exam1} Suppose the value process of an asset is governed by
\begin{equation}%
\begin{cases}
S_{t+1}-S_{t}=\mu S_{t}+\sigma S_{t}(M_{t+1}-M_{t})\\
S_{0}=s>0,
\end{cases}
\label{kn-8}%
\end{equation}
where $s,\mu\in \mathbb{R}^+$, $\sigma\in\mathbb{R}^m\setminus \{0\}$ are given constants. We want to find the optimal time
$\tau^\ast\in\{0,1,...,T\}$ to sell this asset.

We first suppose that there does not exist ambiguity and the risk only comes
from the martingale difference process. This problem can be formulated as
follows
\[
\sup_{0\leq\tau\leq T}E[S_{\tau}].
\]
From (\ref{kn-8}), we know
\[
E[S_{t+1}-S_{t}]=\mu E[S_{t}].
\]
Since $S_{0}=s>0$, we have $E[S_{1}]>S_{0}>0$. It is easy to see that
$E[S_{t+1}]\geq E[S_{t}]>0$. Thus, the optimal time is $\tau^{\ast}=T$, which
implies that the owner is better to hold this asset until the maturity time
$T$.

Now if there exists ambiguity, which can be described by a family of time-consistent
probability measures $\Lambda$ with associated set $\Theta = \{\theta: Q^\theta\in\Lambda\}$. Then an ambiguity averse decision maker wishes to solve
\[
\sup_{\tau}\inf_{\theta\in\Theta}E^{Q^{\theta}}[S_{\tau}].
\]
This problem solved by considering the following RBSDE,
\[
\begin{cases}
\bar{U}_{t}=\bar{U}_{t+1}-\mu \bar U_t + \inf_{\theta\in\Theta}\{Z_t^* \theta_t\}-Z_{t}^{\ast}M_{t+1}+K_{t+1}-K_{t}\\
\bar{U}_{T}=S_{T},\quad \bar{U}_{t}\geq S_{t}\\
(\bar{U}_{t}-S_{t})(K_{t+1}-K_{t})=0.
\end{cases}
\]
By Proposition \ref{propositation fsbsde}, we know
\[
\tau^{\ast}=\inf\{u\leq T;\quad \bar{U}_{u}=S_{u}\}\wedge T.
\]
Hence, ambiguity aversion can encourage ealier stopping.
\end{example}

\section{Applications to pricing of American contingent claims}

It is well-known that the price of an American option corresponds
to the solution of a reflected BSDE, where the information flow is generated
by the Brownian motion \cite{r11,BH}. In this section, we explore this pricing
problem in the discrete time and finite state cases, where the martingale
difference process replace the Brownian motion.

We begin with the classical set-up for discrete time asset
pricing: the basic securities consist of $m+1$ assets $\{S_{t}^{i}\}_{0\leq t\leq
T, i\in\{0,1,...,m\}}$, one of which is a non-risky asset with price process as
follows:
\[
S_{t+1}^{0}-S_{t}^{0}=r_{t}S_{t}^{0},
\]
where $r_{t}$ is the interest rate. The other $k$ risky asset (the stocks) are
traded discretely, of which the price process $S_{t}^{i}$ for one share of
$i$th stock is governed by the linear difference equation
\[
S_{t+1}^{i}-S_{t}^{i}=S_{t}^{i}\Big(b_{t}^{i}+\sum_{j=1}^{m}\sigma_{t}^{i,j}M_{t+1}^{j}\Big)
\]
where $M_{t}=(M_{t}^{1},M_{t}^{2},...,M_{t}^{m})^{\ast}$ is our martingale
difference sequence on $\mathbb{R}^{m}$.

We assume that
\begin{enumerate}[(i)]
 \item The short interest rate $r$ is a predictable process
which is generally nonnegative.
\item The stock appreciation rates $b=(b^{1},b^{2},...,b^{n})^{\ast}$ is a predictable process.
\item The volatility matrix $\sigma=(\sigma^{i,j})$ is a predictable process in $\mathbb{R}^{k\times m}$.
\item There exists a predictable $Q$-vector process $\theta$, called the risk premium, such that
\[
b_{t}-r_{t}\mathbf{1}=\sigma_{t}\theta_{t},\;dt\times dP-a.e..
\]
 where $\mathbf{1}$ is the vector whose every component is $1$. Denote by $\Theta$ the family of all such processes.
\end{enumerate}

We note that each $\theta\in\Theta$ corresponds to a measure where
\[E^{Q^\theta}[S^i_{t+1}-S_t^i|\mathcal{F}_t] = S_{t}^{i}\Big(b_{t}^{i}+\sum_{j=1}^{m}\sigma_{t}^{i,j}\theta_t^j\Big) = S_{t}^{i}r_t\]
and so $\Theta$ is a representation of the equivalent martingale measures of the discounted processes $S_s^i\prod_{s\leq t}(1+r_s)^{-1}$.

\begin{definition}
A predictable process $H=(H^{0},H^{1},...,H^{k})$ is called self-financing if $\langle H_t, S_{t-1}\rangle=\langle H_t, S_t\rangle$, where $S=(S^0, S^1, \ldots, S^k)$. The value $V$ of of the corresponding self-financing portfolio can be
formulated as follows (refer to \cite{r20}):
\[
V_{t}=H^{0}_tS_{t}^{0}+\sum_{i=1}^{k}H^{i}_tS_{t}^{i}=H^0_{t+1}S_{t}^{0}+\sum_{i=1}^{k}H^{i}_{t+1}S_{t}^{i}.
\]
So
\begin{equation}\label{dwe}
 \begin{split}
  V_{t+1}-V_{t}  &  =H^{0}_{t+1}(S_{t+1}^{0}-S_{t}^{0})+\sum_{i=1}^{k} H^{i}_{t+1}(S_{t+1}^{i}-S_{t}^{i})\\
&  =r_{t}V_{t}+\sum_{i=1}^{k}H^{i}_{t+1}S_{t}^{i}\Big(b_{t}^{i}-r_{t}+\sum_{j=1}^{m}\sigma_{t}^{i,j}M_{t+1}^{j}\Big)\\
&  =r_{t}V_{t}-z_t^*(\theta_{t}+ M_{t+1}),
 \end{split}
\end{equation}
where $z_t^j= -\sum_{i=1}^k H^{i}_{t+1}S_{t}^{i}\sigma_{t}^{i,j}$ and $\theta\in\Theta$.
\end{definition}
 We then have the following subreplication result.
\begin{theorem}
 Let $f(t,y,z) = -r y + \inf_{\theta\in\Theta} \{z^*\theta_t\}$. Then the solution $(Y,Z)$ to the BSDE with driver $f$ and terminal value $Y_T=\xi$ is equal to the largest subreplication price of the european contingent claim $\xi$.
\end{theorem}
\begin{proof}
Let $R_t=\prod_{s\leq t}(1+r_s)^{-1}$. By standard duality results (see \cite{r20}), we know that the largest subreplication value of $\xi$ is given by $\inf_{Q\in\Lambda}E^Q[\xi R_T/R_t|\mathcal{F}_t]$, where $\Lambda$ is the family of equivalent martingale measures for the discounted processes $S_t^iR_t$. By construction $\Lambda = \{Q^\theta: \theta\in\Theta\}$, and so by Proposition \ref{BSDEinfexp}, this price is given by the solution to the stated BSDE.
\end{proof}
\begin{remark}
In a similar way, we can obtain the minimal superreplication price as the solution of the BSDE with driver $f(t,y,z) = -r y + \sup_{\theta\in\Theta} \{z^*\theta_t\}$, by considering subreplication of $-\xi$. Note that (\ref{dwe}) then shows that, as expected, the self-financing portfolios have the same subreplication and superreplication prices, as (\ref{dwe}) holds for all $\theta\in\Theta$.
\end{remark}

Let us consider the valuation problem of an American contingent
claim with possible payoffs $\{\xi_{t}\}_{0\leq t\leq T}$. The holder can exercise only once,  at a stopping time time $\tau \in\{0,1,...,T\}$. For notational convenience, we define
\[\begin{split}
\bar f(t,y,z) &= -r y + \sup_{\theta\in\Theta} \{z^*\theta_t\}, \\
\underline f(t, y, z) &= -r y + \inf_{\theta\in\Theta} \{z^*\theta_t\},
  \end{split}
\]
and we assume that this $\sup$ and $\inf$ are attained.

It is well known that this kind of claim cannot
be replicated by a self-financing portfolio, and that it is necessary to introduce
self-financing super-strategies with a cumulative consumption process.

\begin{definition}
A self-financing super-strategy is a vector process $(V,H,K),$ where $V$ is
the value process, $H$ is the portfolio process and $K$ is the cumulative
consumption process, such that
\[
V_{t+1}-V_{t}=\langle H_t, S_{t+1}-S_t\rangle -(K_{t+1}-K_{t}),
\]
where $K$ is an increasing adapted process with $K_{0}=0$. Equivalently, it is a process such that $V_{t} \geq E^Q[V_{t+1}|\mathcal{F}_t]$ for all $Q\in\Lambda$. If $-V$ is a super-strategy, then we say that $V$ is a sub-strategy.
\end{definition}

\begin{definition}
Given a payoff process $\{\xi_{t}\}_{t\in\{0,1,...,T\}}$, a
super-strategy is called a superreplication strategy if
\[
V_{t}\geq\xi_{t} \text{ for all }t\in\{0,1,...,T\},P-a.s.
\]
in which case $V$ is called a superreplication price. The value $\inf V_t$, where the infimum is taken over all superreplication prices, is called the minimal superreplication price.

In a similar way, with the inequality reversed, we define sub-replication strategies (which are sub-strategies with $V_{t}\geq\xi_{t}$ for all $t$) and the maximal subreplication price.
\end{definition}

\begin{theorem}
Let $(Y,Z, K)$ be the solution to the RBSDE with driver $\bar f$, terminal value $\xi_T$ and lower barrier $\xi_t$. Then $Y_t$ is equal to the smallest superreplication price of the American contingent claim with payoff $\{\xi_t\}$. Similarly, the RBSDE with driver $\underline f$ yields the largest subreplication price for the claim.
\end{theorem}
\begin{proof}
We consider the superreplication price only, the subreplication price is similar. Recall from Proposition \ref{propositation fsbsde} that if $(Y, Z, K)$ is the solution of the RBSDE with driver $f$ and lower barrier $\xi_\tau$, then
\begin{equation}\label{eqYAmerican}
Y_t = \sup_{\tau} E\Big[\sum_{t\leq s< \tau} \bar f(s, Y_s, Z_s) + \xi_\tau\Big|\mathcal{F}_t\Big].
\end{equation}
The one-step dynamics for $Y$ are
\[\begin{split}
   Y_t-(K_{t+1}-K_t) &= Y_{t+1} + f(t, Y_t, Z_t) - Z_t^* M_{t+1}\\
&= Y_{t+1} -r_t Y_t + \sup_{\theta\in\Theta} \{Z_t^* \theta_t\} - Z_t^* M_{t+1}\\
&= \sup_{Q\in\Lambda} E^Q[Y_{t+1}R_{t+1}/R_t|\mathcal{F}_t].
  \end{split}\]
As $K$ is an increasing process we see that $Y$ corresponds to a super-strategy, and so $Y_t$ is a superreplication price for $\xi$. Conversely, if $Y'$ is a superreplication price for $\{\xi_t\}$, then from (\ref{dwe}) there would exist a process $z'$ and an increasing consumption process $K'$ such that
\[Y'_{t+1}-Y'_{t}  =r_{t}Y'_{t}-(z_t')^*(\theta_{t}+ M_{t+1}) + K'_{t+1}-K'_t\]
and $Y'_t\geq \xi_t$. By the comparison theorem for RBSDEs, this implies $Y'_t\geq Y_t$, so $Y$ is the minimal superreplication price for $\{\xi_t\}$.
\end{proof}

\begin{remark}
By replacing $\xi_t$ with $-\xi_t$, we can consider the perspective of the seller of a claim, who will have to pay out when it is exercised. After changing sign again, this results in the supremum over $\tau$ in (\ref{eqYAmerican}) being replaced with an infimum, as the seller cannot control the exercise time of the option.
\end{remark}

\textbf{Acknowledgements} The authors would like to thank Prof. Shige Peng and
Frank Riedel for some useful conversations in the stochastic game workshop in
China (November 2012).


\bigskip

\begin{thebibliography}{99}                                                                                               %

\bibitem {BHM}{ \textsc{Bahlali, K.}, \textsc{Hamadene, S.} and
\textsc{Mezerdi, B.} (2005). \textit{ BSDEs with two reflecting barriers and
quadratic growth coefficient}. Stochastic Processes and their Applications
\textbf{115} 1107-1129.}

\bibitem {r1}{ \textsc{Bouchard, B.} and \textsc{Touzi, N.} (2004).
\textit{Discrete-time approximation and Monte-Carlo simulation of backward
stochastic differential equations}. Applied Stochastic Processes. \textbf{111}
175--206. }

\bibitem {r2}{ \textsc{Briand, P.} and \textsc{\ Delyon, B.} and
\textsc{\ M}{\small E}\textsc{min, J.} (2001). \textit{\ Donsker-type theorem
for BSDEs}. Electronic Communications in Probability. \textbf{\ 6} 1--14. }

\bibitem {BH}{ \textsc{Buckdahn, R.} and \textsc{Hu, Y.} (1998).
\textit{Pricing of American contingent claims with jump stock price and constrained portfolios}. Math. Oper. Res. \textbf{23}
177--203. }

\bibitem {r3}{ \textsc{Chen, Z. } and \textsc{Epstein, L.} (2002).
\textit{\ Ambiguity, risk, and asset returns in continuous time}.
Econometrica. \textbf{70(4)} 1403--1443. }

\bibitem {r5}{ \textsc{Cheng, X. and Riedel, F. (2013)} Optimal
Stopping Under Ambiguity in Continuous Time, Mathematics and Financial
Economics 7(1) 29-68. }

\bibitem {CohenAggregation}{ \textsc{Cohen, S. N. (2012)} \textit{Quasi-sure analysis, aggregation and dual representations of sublinear expectations in general spaces}.  Electronic Journal of Probability. \textbf{17} Article 62}

\bibitem {SCGenDom}{ \textsc{Cohen, S. N. (2012)} \textit{Representing filtration
consistent nonlinear expectations as $g$ expectation in general probability spaces}. Stochastic
 Processes and their Applications. \textbf{122}(4) 257--269.}
%
%

\bibitem {r8}{ \textsc{Cohen, S. N.} and \textsc{Elliott, R. J. }
(2010). \textit{\ A general theory of finite state backward stochastic
difference equations }. Stochastic Processes and their Applications.
\textbf{\ 120 } 442--466. }

\bibitem {CE}{ \textsc{Cohen, S. N.} and \textsc{Elliott, R. J.}
(2011). \textit{\ Backward stochastic difference equations and nearly
time-consistent nonlinear expectations.} SIAM J. Control Optim. 49 (1)
125--139.}

\bibitem {r9}{ \textsc{ Cohen, S. N.}, \textsc{ Elliott, R. J.},
\textsc{Pardoux, E.}, and \textsc{Pearce, C. E. M.}(2008). \textit{\ A Ring
Isomorphism and corresponding Pseudoinverses}. arXiv:0810.0093. }
%


\bibitem {CH}{ \textsc{Cohen, S. N.} and \textsc{Hu, Y.}
(2012). \textit{\ Ergodic BSDEs driven by Markov Chains.} On arXiv:1207.5680.}


\bibitem {CH1}{ \textsc{Cohen, S. N.} and \textsc{Hu, Y.}
(2013). \textit{\ Undiscounted Markov chain BSDEs to stopping times.} On arXiv:1302.4637.}

\bibitem {CHMP}{ \textsc{Coquet, F.}, \textsc{Hu, Y.},
\textsc{Memin J.} and \textsc{Peng, S.} (2002). \textit{\ Filtration
Consistent Nonlinear Expectations and Related g-Expectations}. Probab. Theory
Relat. Fields 123 1-27.}

\bibitem {CK}{ \textsc{Cvitanic J.} and \textsc{Karatzas I.} (1996)
\textit{\ Backward stochastic diffential equations with reflection and Dynkin
games}. Ann. Probab. 24, 2024-2056.}

\bibitem{Delbaen}{ \textsc{Delbaen F.} (2006)
\textit{ The Structure of $m$-Stable sets and in particular the set of risk neutral measures} in In Memoriam Paul-Andr\'e Meyer (ed. Emery and Yor), Springer.}

\bibitem {r10}{ \textsc{\ El Karoui, N.}, \textsc{\ Kapoudjan, C.},
\textsc{Pardoux, E.}, \textsc{Peng, S.} and \textsc{Quenez, M.C.}(1997).
\textit{\ Reflected solutions of backward SDE's, and related obstacle problem
for PDE's}. Annals of Probability. \textbf{\ 25(2) } 702--737. }

\bibitem {r11}{ \textsc{\ El Karoui, N. }, \textsc{Pardoux, E. }
and \textsc{Quenez, M. C.} (1997). \textit{\ Reflected backward SDEs and
American options}. Numerical Methods in Finance [M], Cambridge: Publications
of the Newton Institute, Cambridge University Press. 215--231.}

\bibitem {EPQ}{ \textsc{\ El Karoui, N. }, \textsc{Peng, S. } and
\textsc{Quenez, M. C.} (1997). \textit{Backward stochastic differential
equations in finance. Math. Finance. 7(1) \ 1-71.}}

\bibitem {H}{ \textsc{Hamadene S.} (2002). \textit{Reflected BSDEs
with discontinuous barrier and application}. Stochastics Rep. 74, 571-596.}

\bibitem {HLM1997}{ \textsc{Hamadene S.}, \textsc{Lepeltier J.} and
\textsc{Matoussi A.} (1997). \textit{Double barriers reflected backward SDEs
with continuous coefficients.} Pitman Research Notes in Mathematics Series
364, 115-128, Longman (Eds : N.E-K. \& L.M.)}

\bibitem {HT}{ \textsc{Hu, Y.} and \textsc{Tang, S.} (2010).
\textit{Multi-dimensional BSDE with oblique reflection and optimal switching.}
Probab. Theory Related Fields 147(1-2) 89--121.}

\bibitem {LX}{ \textsc{Lepeltier J.} and \textsc{Xu M.} (2005).
\textit{Penalization method for reflected backward stochastic differenctial
equations with one R.C.L.L. barrier.} Statistics \& Probability Letters 75,
58-66.}

\bibitem {LM}{ \textsc{Lepeltier, J.} and \textsc{Xu, M.} (2007).
\textit{Reflected backward stochastic differential equations with two RCLL
barriers.} ESAIM Probab. Stat. 11 3--22.}

\bibitem {r15}{ \textsc{\ Ma, J. }, \textsc{Protter, P.},
\textsc{Martin, S.} and \textsc{Torres, S.} (2002). \textit{\ Numerical method
for backward stochastic differential equations}. Annals of Probability.
\textbf{\ 12(1)} 302--316. }

\bibitem {Matoussi}{ \textsc{Matoussi, A.} (1997). \textit{Reflected
solutions of BSDEs with continuous coefficient.} Statistics and Probability
Letters 34, 347-354.}

\bibitem{NutzSoner}{ \textsc{Nutz, M.} and \textsc{Soner, H.M.} (2010).
\textit{Superhedging and dynamic Risk Measures under volatility uncertainty} Siam J. Control and Optimization. \textbf{50}(4) 2065-–2089}

\bibitem {r16}{ \textsc{\ }}O{ \textsc{ksendal, B.}
(2000). \textit{\ Stochastic Diffrential Equations-An Introduction with
Applications}, Fifth ed. Springer-Verlag. }

\bibitem {r17}{ \textsc{\ Pardoux, E.} and \textsc{Peng, S.}
(1990). \textit{Adapted solution of a backward stochastic differential
equation}. Systems and Control Letters. \textbf{14} 55--61. }

\bibitem {r18}{ \textsc{Peng, S.}(1997). \textit{\ Backward
Stochastic differential equations and related g-expectations}. Backward
Stochastic Differential Equations, \textsc{El Karoui, N.} and \textsc{Mazliak,
L.} (eds.). Pitman Research Notes in Mathematics series, Longman Harlow.
\textbf{\ 364} 141--159. }

\bibitem {r19}{ \textsc{Peng, S. }(1999). \textit{Monotonic limit
theorem of BSDE and nonlinear decomposition theorem of Doob-Mayer}.
Probability Theory and Related Fields. \textbf{113} 473--499. }

\bibitem {r20}{ \textsc{Pliska, S. R.} (1997). \textit{Introduction to Mathematical Finance:
Discrete Time Model.} Black-well Publishing, New York. 114--115.}

\bibitem {r21}{ \textsc{Revuz, D.} and \textsc{Yor, M.} (1994).
\textit{Continuous Martingales and Brownian Motion}. Springer, New York.
239--240. }

\bibitem {r22}{ \textsc{Riedel, F. }(2009). \textit{Optimal
Stopping with Multiple Priors}. Econometrica. \textbf{77(3)} 857--908. }

\bibitem {CherSta}{ \textsc{Stadje, M. and Cheridito, P. (2011)}
 BS$\Delta$Es and BSDEs with non-Lipschitz drivers:
comparison, convergence and robustness, Bernoulli, to appear.}
\end{thebibliography}
\end{document}